\documentclass[11pt]{amsart}
\usepackage{geometry}
\usepackage{graphicx}
\usepackage{amssymb}
\usepackage{epstopdf}
\usepackage{epsfig}
\usepackage{amsmath}
\usepackage{latexsym}
\usepackage{amssymb}
\usepackage{mathrsfs}
\usepackage{pdflscape}
\usepackage{hyperref}
\newtheorem{theorem}{Theorem}[section]
\newtheorem{proposition}{Proposition}[section]
\newtheorem{corollary}{Corollary}[section]
\newtheorem{lemma}{Lemma}[section]
\newtheorem{conjecture}{Conjecture}
\newtheorem{example}{Example}

\theoremstyle{remark}

\title{Gram Determinant of Planar Curves}
\author{J\'ozef H. Przytycki and Xiaoqi Zhu}
\begin{document}
\maketitle

\begin{abstract}
We investigate the Gram determinant of the bilinear form based on curves 
in a planar surface, with a focus on the disk with two holes. 
We prove that the determinant based on $n-1$ curves divides the determinant 
based on $n$ curves. Motivated by the work on Gram determinants based 
on curves in a disk and curves in an annulus (Temperley-Lieb algebra of 
type $A$ and $B$, respectively), we calculate several examples of the Gram 
determinant based on curves in a disk with two holes and advance 
conjectures on the complete factorization of Gram determinants.
\end{abstract}
\section{Introduction}
Let $F_{0,0}^n$ be a unit disk with $2n$ points on its boundary. Let $\textbf{B}_{n,0}$ be the set of all possible diagrams, up to deformation, in $F_{0,0}^n$ with $n$ non-crossing chords connecting these $2n$ points. It is well-known that $|\textbf{B}_{n,0}|$ is equal to the $n^\text{th}$ Catalan number $C_n = \frac{1}{n+1}{2n \choose n}$ \cite{Sta}. Accordingly, we will call $\textbf{B}_{n,0}$ the set of \textbf{Catalan states}.\\

We will now generalize this setup. Let $F_{0,k} \subset D^2$ be a plane surface with $k+1$ boundary components. $F_{0,0} = D^2$, and for $k \ge 1$, $F_{0,k}$ is equal to $D^2$ with $k$ holes. Let $F_{0,k}^n$ be $F_{0,k}$ with $2n$ points, $a_0, \ldots, a_{2n-1}$, arranged counter-clockwise along the outer boundary, cf. Figure \ref{fig1}.

\begin{figure}[htbp]
\begin{center}
\includegraphics[scale=0.9]{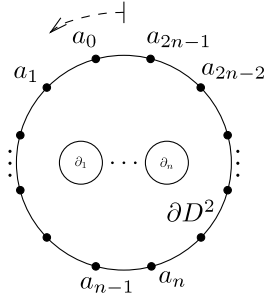}
\end{center}
\caption{Throughout the paper, we number the points counter-clockwise beginning at the top of outer boundary. We label and differentiate between the holes.}
\label{fig1}
\end{figure}

Let $\textbf{B}_{n,k}$ be the set of all possible diagrams, up to equivalence, in $F_{0,k}^n$ with $n$ non-crossing chords connecting these $2n$ points. We define equivalence as follows: for each diagram $b \in \textbf{B}_{n,k}$, there is a corresponding diagram $\gamma(b) \in \textbf{B}_{n,0}$ obtained by filling the $k$ holes in $b$. We call $\gamma(b)$ the underlying Catalan state of $b$ (cf. Figure \ref{catalan}).
\begin{figure}[htbp]
\begin{center}
\includegraphics[scale=1]{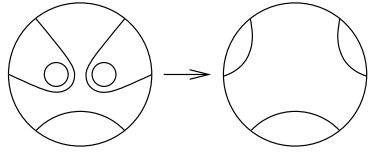}
\end{center}
\caption{$b \mapsto \gamma(b)$}
\label{catalan}
\end{figure}
In addition, a given diagram in $F_{0,k}^n$ partitions $F_{0,k}$ into $n+1$ regions. Two diagrams are equivalent if and only if they have the same underlying Catalan state and the labeled holes are distributed in the same manner across regions. Accordingly, $|\textbf{B}_{n,k}| = (n+1)^{k-1}{2n \choose n}$. We remark that in the cases $k=0$ and $k=1$, two diagrams are equivalent if they are homotopic, but for $k>2$, this is not the case (for an example, 
see Figure \ref{fig3}).\\

In this paper, we define a pairing over $\textbf{B}_{n,k}$ and investigate 
the Gram matrix of the pairing. This concept is a generalization of a problem 
posed by W. B. R. Lickorish for type $A$ (based on a disk, i.e. $k=0$) 
Gram determinants, and Rodica Simion for type $B$ (based on an annulus, i.e. 
$k=1$) Gram determinants, cf. \cite{Lic-1, Lic-2}, \cite{Sch, Sim}. 
Significant research has been completed for the Gram determinants for type 
$A$ and $B$. In particular, P. Di Francesco and B. W. Westbury gave a closed 
formula for the type $A$ Gram determinant \cite{DiF}, \cite{Wes}; 
a complete factorization of the type $B$ Gram determinant was conjectured 
by Gefry Barad and a closed formula was proven by Q. Chen and J. H. Przytycki 
\cite{Che} (see also \cite{MS}). The type $A$ Gram determinant was used by Lickorish to find an elementary construction of Reshetikhin-Turaev-Witten invariants of oriented closed 3-manifolds.\\

We specifically investigate the Gram determinant $G_n$ of the bilinear form defined over $\textbf{B}_{n,2}$ and prove that $\det G_{n-1}$ divides $\det G_n$ for $n > 1$. Furthermore, we investigate the diagonal entries of $G_n$ and give a method for computing terms of maximal degree in $\det G_n$. We conclude the paper by briefly discussing generalizations of the Gram determinant and presenting some open questions.

\section{Definitions for $\textbf{B}_{n,2}$}
Consider $F_{0,2}^n$, a unit disk with two holes, along with $2n$ points along the outer boundary. Denote the holes in $F_{0,2}^n$ by $\partial_{X_1}$ and $\partial_{Y_1}$, or more simply, just $X_1$ and $Y_1$. To differentiate between the two holes, we will always place $X_1$ to the left and $Y_1$ to the right if labels are not present.

\begin{figure}[htbp]
\begin{center}
\includegraphics{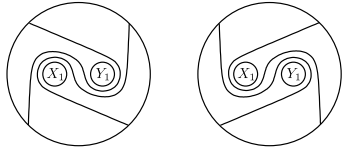}
\end{center}
\caption{Two non-isotopic but equivalent diagrams in $F_{0,2}^2$. 
They correspond to the same state in $\textbf{B}_2$. A complete specification of $\textbf{B}_2$ can be found in the Appendix.}
\label{fig3}
\end{figure}

Let $\textbf{B}_n := \textbf{B}_{n,2} := \{b_1^n, \ldots, b_{(n+1){2n \choose n}}^n\}$, the set of all possible diagrams, up to equivalence in $F_{0,2}^n$ with $n$ non-crossing chords connecting these $2n$ points. For simplicity, we will often use $b_i$ instead of $b_i^n$, when the number of points along the outer boundary can be inferred from context.

\begin{figure}[htbp]
\begin{center}
\includegraphics{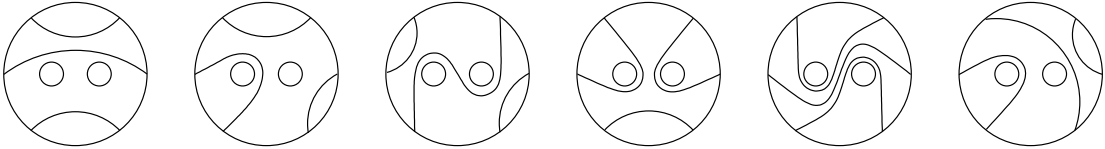}
\end{center}
\caption{Pictorial representations of six states $\{b_1, b_2, b_3, b_4, b_5, b_6\} \subset \textbf{B}_3$. We stress that this is not a natural ordering of states in $\textbf{B}_3$.}
\label{fig2}
\end{figure}

Let $X_2$ and $Y_2$ be the inversions\footnote{Inversion is an involution 
defined on the sphere $\mathbb{C} \cup \infty$ by $z \leftrightarrow \frac{z}{|z|^2}$.} of $X_1$ and $Y_1$, respectively, with respect to the unit disk, and let $\mathcal{S} = \{X_1, X_2, Y_1, Y_2\}$. Given $b_i \in \textbf{B}_n$, let ${b_i}^*$ denote the inversion of $b_i$. Given $b_i, b_j \in \textbf{B}_n$, we glue $b_i$ with ${b_j}^*$ along the outer boundary, respecting the labels of the marked points. $b_i$ and $b_j$ each contains $n$ non-crossing chords, so $b_i \circ {b_j}^*$ can have at most $n$ closed curves. The resulting diagram, denoted by $b_i \circ {b_j}^*$, is then a set of up to $n$ closed curves in the 2-dimensional sphere ($D^2 \cup {(D^2)}^*$) with four holes: $X_1, X_2, Y_1, Y_2$ (we disregard the outer boundary, $\partial D^2$). Each closed curve partitions the set $\mathcal{S}$ into two sets. Two closed curves are of the same type if they partition $\mathcal{S}$ the same way. For each $b_i \circ {b_j}^*$, there are then up to eight types of disjoint closed curves, whose multiplicities we index by the following variables:
\begin{eqnarray*}
n_d &= &\text{the number of curves with $\{X_1, X_2, Y_1, Y_2\}$ on the same side}\\
n_{x_1} &= &\text{the number of curves that separate $\{X_1\}$ from $\{X_2, Y_1, Y_2\}$}\\
n_{x_2} &= &\text{the number of curves that separate $\{X_2\}$ from $\{X_1, Y_1, Y_2\}$}\\
n_{y_1} &= &\text{the number of curves that separate $\{Y_1\}$ from $\{X_1, X_2, Y_2\}$}\\
n_{y_2} &= &\text{the number of curves that separate $\{Y_2\}$ from $\{X_1, X_2, Y_1\}$}\\
n_{z_1} &= &\text{the number of curves that separate $\{X_1, X_2\}$ from $\{Y_1, Y_2\}$}\\
n_{z_2} &= &\text{the number of curves that separate $\{X_1, Y_1\}$ from $\{X_2, Y_2\}$}\\
n_{z_3} &= &\text{the number of curves that separate $\{X_1, Y_2\}$ from $\{X_2, Y_1\}$}
\end{eqnarray*}

Let $R:= \mathbb{Z} [d,x_1,x_2,y_1,y_2,z_1,z_2,z_3]$, and $R\textbf{B}_n$ be the free module over the ring $R$ with basis $\textbf{B}_n$. We define a bilinear form $\langle , \rangle : R\textbf{B}_n \times R\textbf{B}_n \rightarrow R$ by: $$\langle b_i,b_j\rangle = d^{n_d}{x_1}^{n_{x_1}}{x_2}^{n_{x_2}}{y_1}^{n_{y_1}}{y_2}^{n_{y_2}}{z_1}^{n_{z_1}}{z_2}^{n_{z_2}}{z_3}^{n_{z_3}}$$
$\langle b_i, b_j \rangle$ is a monomial of degree at most $n$.  Some examples of paired diagrams and their corresponding monomials, using examples from 
Figure \ref{fig2}, are given in Figure \ref{fig4}.

\begin{figure}[htbp]
\begin{center}
\includegraphics[scale=0.9]{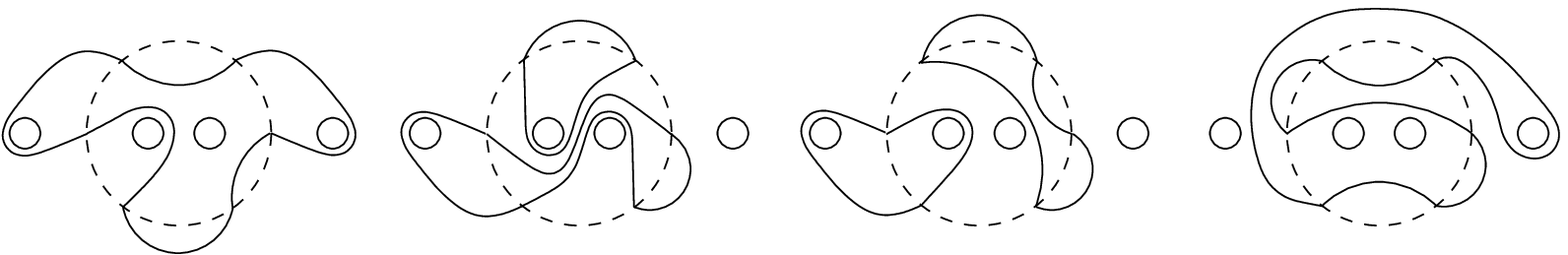}
\end{center}
\caption{From left to right:
$$\langle b_2, b_4\rangle = x_1 \quad \langle b_5, b_2\rangle = x_1x_2 \quad \langle b_6, b_2\rangle = dz_1 \quad \langle b_1, b_3\rangle = x_2$$}
\label{fig4}
\end{figure}

Let
$$G_n = \begin{pmatrix}g_{ij}\end{pmatrix} = {\begin{pmatrix}\langle b_i, b_j \rangle\end{pmatrix}}_{1 \le i,j \le (n+1){2n \choose n}}$$
be the Gram matrix of the pairing on $\textbf{B}_n$. For example,
$$
G_1 = \begin{bmatrix}
d&y_2&x_2&z_2\\
y_1&z_1&z_3&x_1\\
x_1&z_3&z_1&y_1\\
z_2&x_2&y_2&d
\end{bmatrix}
\text{ up to ordering of }\textbf{B}_1 \text{ and}
$$
\begin{eqnarray*}
\det G_1 &= &((d+z_2)(z_1+z_3)-(x_1+y_1)(x_2+y_2))
\\&&((d-z_2)(z_1-z_3)-(x_1-y_1)(x_2-y_2)).
\end{eqnarray*}

\begin{figure}[htbp]
\begin{center}
\includegraphics[scale=0.6]{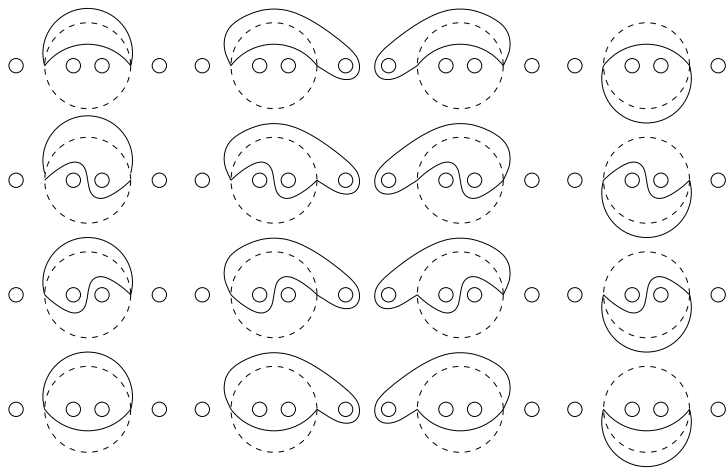}
\end{center}
\caption{A pictorial representation of curves used to define $G_1$}
\label{fig5}
\end{figure}

We remark that for $b_i, b_j \in \textbf{B}_n$, $\langle b_j, b_i \rangle$ 
can be obtained by taking $b_i \circ {b_j}^*$ and interchanging the roles 
of $X_1$ and $Y_1$ with $X_2$ and $Y_2$, respectively. Let $h_t$ be an 
involution on the entries of $G_n$ which interchanges the variables 
$x_1$ with $x_2$ and $y_1$ with $y_2$. 
It follows that $\langle b_i, b_j\rangle = h_t(\langle b_j, b_i \rangle)$.  
The transpose matrix is then given by:
$$^tG_n = \begin{pmatrix}h_t(\langle b_i, b_j \rangle)\end{pmatrix}$$
We note that the variables $d, z_1, z_2, z_3$ are preserved by $h_t$ 
(cf. Theorem 3.2(4)).\\

We can define more generally: given 
$A = \{b_{n_1}, b_{n_2}, \ldots, b_{n_p}\} \subseteq \textbf{B}_n$ and 
$B = \{b_{m_1}, b_{m_2}, \ldots, b_{m_q}\} \subseteq \textbf{B}_n$, 
let $\langle A, B\rangle$ be an $p \times q$ submatrix of $G_n$ given by:
$$\langle A, B\rangle = {\begin{pmatrix}\langle b_{n_i}, b_{m_j} 
\rangle\end{pmatrix}}_{1 \le i \le p, 1 \le j \le q}$$
For example, we can express the matrix $G_n$ as $\langle \textbf{B}_n, \textbf{B}_n\rangle$. The $i^{\text{th}}$ row of $G_n$ can be written as $\langle b_i, \textbf{B}_n\rangle$.\\

This paper is mostly devoted to exploring possible factorizations of $\det G_n$, and is the first step toward computing $\det G_n$ in full generality, which we conjecture to have a nice decomposition.\\

Let $i_0: \textbf{B}_n \rightarrow \textbf{B}_{n+1}$ be the \textbf{embedding} map defined as follows: for $b_i \in \textbf{B}_n$, $i_0(b_i) \in \textbf{B}_{n+1}$ is given by adjoining to $b_i$ a non-crossing chord close to the outer boundary that intersects the outer circle at two points between $a_0$ and $a_{2n-1}$, cf. upper part of Figure \ref{fig7}.\\

We will also use a generalization of $i_0$, for which we need first the following definition. For any real number $\alpha$, consider the homeomorphism $r_\alpha: \mathbb{C} \rightarrow \mathbb{C}$ on the annulus $R' \le |z| \le 1$, which we call the \textbf{$\alpha$-Dehn Twist}, defined by:
$$r_\alpha(z) = ze^{i\alpha\left(1-(1-|z|)/(1-R')\right)}$$
Note that $r_\alpha(z) = z$ as $|z| = R'$. Therefore, we can extend the domain of $r_\alpha$ to $D^2$ by defining $r_\alpha(z) = z$ for $0 \le |z| \le R'$. Fix $R'$ such that a circle of radius $R'$ encloses $X_1$ and $Y_1$. Then $r_\alpha$ acts on $b_i \in \textbf{B}_n$ as a clockwise rotation of a diagram close to the outer boundary.\\

\begin{figure}[htbp]
\begin{center}
\includegraphics{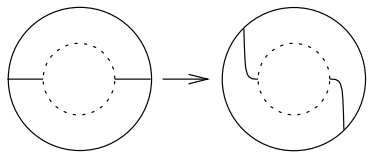}
\end{center}
\caption{A $\pi/4$-Dehn Twist. Note that $r_{2\pi}(b_i) = b_i$ (cf. Figure \ref{fig3}).}
\label{fig6}
\end{figure}

Consider the \textbf{$k$-conjugated embedding} $i_k: \textbf{B}_n \rightarrow \textbf{B}_{n+1}$ defined by:
$$i_k(b_i) = {r_{\pi/n+1}}^ki_0{r_{\pi/n}}^{-k}(b_i)$$
Intuitively, if for $b_i \in \textbf{B}_{n+1}$ there exists $b_j \in \textbf{B}_n$ such that $i_k(b_j) = b_i$, then $b_i$ is composed of $b_j$ and a non-crossing chord close to the outer boundary connecting $a_k$ and $a_{k-1}$\footnote{Throughout this paper, we use $a_k$ and $a_{k-1}$ to denote two adjacent points along the outer boundary, where $k$ is taken modulo $2n$.}, Figure \ref{fig7}.

\begin{figure}[htbp]
\begin{center}
\includegraphics{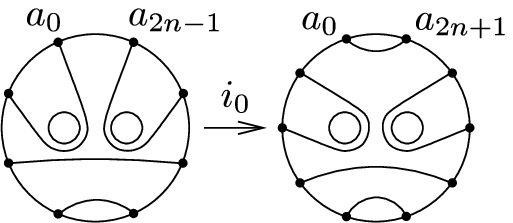}
\\
\
\\
\includegraphics{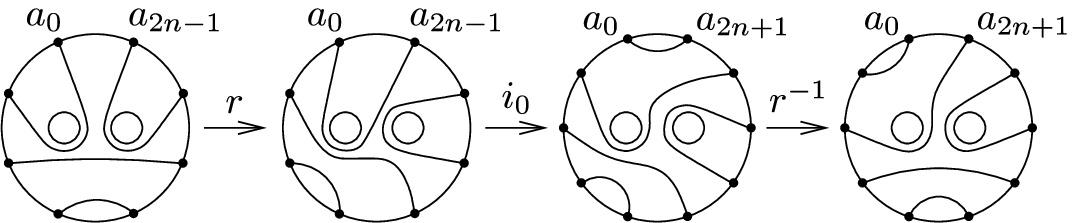}
\end{center}
\caption{An embedding $b_i \mapsto i_0(b_i)$, top; a 1-conjugated embedding $b_i \mapsto i_1(b_i)$, bottom; $b_i \in \textbf{B}_4$.}
\label{fig7}
\end{figure}

For every $b_i \in \textbf{B}_n$, let $p_k(b_i)$ be the diagram obtained by gluing to $b_i$ a non-crossing chord connecting $a_k$ and $a_{k-1}$ outside the circle, and pushing the chord inside the circle. The properties of $p_k$ will be explored in greater detail in Section 4. We conclude this section with a basic identity linking $i_0$ and $p_0$:

\begin{proposition}
For any $b_i \in \textbf{B}_n$, $b_j \in \textbf{B}_{n-1}$, $b_i \circ {i_0(b_j)}^* = p_0(b_i) \circ {b_j}^*$.
\end{proposition}

\section{Basic Properties of Gram Determinant}
In this section, we prove basic properties of $\det G_n$. In particular, we show that the determinant of $G_n$ is nonzero.
\begin{lemma}
$\langle b_i, b_j \rangle$ is a monomial of maximal degree if and only if $\gamma(b_i) = \gamma(b_j)$.
\end{lemma}

\begin{proof}
$b_i \circ {b_j}^*$ has $n$ closed curves if and only if each closed curve is formed by exactly two arcs, one in $b_i$ and one in ${b_j}^*$. Hence, any two points connected by a chord in $b_i$ must also be connected by a chord in $b_j$, so $\gamma(b_i) = \gamma(b_j)$.
\end{proof}

\begin{theorem}
$\det G_n \neq 0$ for all integers $n \ge 1$.
\end{theorem}
\begin{proof}
Assume $\langle b_i, b_j \rangle$ is a monomial of maximal degree consisting 
only of the variables $d$ and $z_1$. Because $\gamma(b_i) = \gamma(b_j)$ by 
Lemma 3.1, it follows that any two points connected in $b_i$ are also 
connected in $b_j$. Each connection in $b_i$ can be drawn in four different 
ways with respect to $X$ and $Y$, since there are two ways to position 
the chord relative to each hole. 
Because $\langle b_i, b_j \rangle$ is assumed to consist only of the 
variables $d$ and $z_1$, 
it follows that each pair of arcs that form a closed curve in $b_i \circ {b_j}^*$ either separates $\{X_1, X_2\}$ from $\{Y_1, Y_2\}$ or has $\{X_1, X_2, Y_1, Y_2\}$ on the same side of the curve. One can check each of the four cases to see that this condition implies that any two arcs that form a closed curve in $b_i \circ {b_j}^*$ must be equal, so $b_i = b_j$. Using Laplacian expansion, this implies that the product of the diagonal of $G_n$ is the unique summand of degree $ n(n+1){\binom{2n}n}$ in $\det G_n$ consisting only of the variables $d$ and $z_1$.
\end{proof}

We need the following notation for the next theorem: 
let $f: \alpha_1 \leftrightarrow \alpha_2$ denote a function $f$ which acts 
on the entries of $G_n$ by interchanging variables $\alpha_1$ with $\alpha_2$. 
We can extend the domain of $f$ to $G_n$. Let $f(G_n)$ denote the matrix 
formed by applying $f$ to all the individual entries of $G_n$.\\

Let $h_1,h_2,h_3$ be involutions acting on the entries of $G_n$ with 
the following definitions:
\begin{enumerate}
\item $h_1: x_1 \leftrightarrow y_1 \quad z_1 \leftrightarrow z_3$
\item $h_2: x_2 \leftrightarrow y_2 \quad z_1 \leftrightarrow z_3$
\item $h_3=h_1h_2: x_1 \leftrightarrow y_1 \quad x_2 \leftrightarrow y_2$
\item $h_t: x_1 \leftrightarrow x_2 \quad y_1 \leftrightarrow y_2$
\end{enumerate}

\begin{theorem}
\
\begin{enumerate}
\item $\det h_1(G_1) = -\det G_1$, and
 for $n>1$, $\det h_1(G_n) =\det G_n$.
\item $\det h_2(G_1) = -\det G_1$, and
 for $n>1$, $\det h_2(G_n) =\det G_n$.
\item $\det h_3(G_n) =\det G_n$.
\item $\det h_t(G_n) =\det G_n$.
\end{enumerate}
\end{theorem}

\begin{proof}
For (1), note that $h_1(G_n)$ corresponds to exchanging 
the positions of the holes $X_1$ and $Y_1$ for all $b_i \in \textbf{B}_n$. 
${b_j}^*$ is unchanged, so $h_1$ can be realized by a permutation of rows. 
For states where $X_1$ and $Y_1$ lie in the same region, 
their corresponding rows are unchanged by $h_1$. 
The number of such states is given by $\frac{1}{n+1}|\textbf{B}_n|$. 
Thus, the total number of row transpositions is equal to
$$\frac{1}{2}\left(|\textbf{B}_n| - \left(\frac{1}{n+1}\right)|\textbf{B}_n|\right) = 
\frac{n}{2}{\binom{2n}n} = \left(\frac{n(n+1)}{2}\right)C_n$$
where $C_n = \frac{1}{n+1}{\binom{2n}n}$. It is a known combinatorial fact 
that $C_n$ is odd if and only if $n=2^m-1$ for some $m$, \cite{De}. 
Hence, $C_n$ is odd implies that
$$\frac{n(n+1)}{2} = \frac{2^m(2^m-1)}{2} = 2^{m-1}(2^m-1)$$
which is even for all $m>1$. Thus, $h_1(G_n)$ can be obtained from 
$G_n$ by an even permutation of rows for $n > 1$, so 
$\det h_1(G_n) = \det G_n$. $h_1(G_1)$ is given by an odd number of row 
transpositions on $G_1$, so $\det h_1(G_1) = -\det G_1$. \\

(2) can be shown using the same method of proof as before. $h_2(G_n)$ corresponds to 
exchanging the positions of the holes $X_2$ and $Y_2$ for all $b_i \in \textbf{B}_n$. $h_2$ can thus be realized by a permutation of columns, and the rest of the proof follows in a similar fashion as the previous one. Since $h_2(G_n)$ can be obtained from $G_n$ by an even permutation of columns for $n>1$, $\det h_2(G_n) = \det G_n$. $h_2(G_2)$ is given by an odd number of column transpositions on $G_1$, so $\det h_2(G_1) = -\det G_1$, which proves (2).\\

Since $h_3 = h_1h_2$, it follows immediately that $\det h_3(G_n) = \det G_n$ for $n>1$. The sum of two odd permutations is even, so the equality also holds for $n=1$, which proves (3). (4) follows because $\det h_t(G_n) = \det {}^tG_n = \det G_n$.
\end{proof}
\begin{theorem}
$\det G_n$ is preserved under the following involutions on variables:
\begin{enumerate}
\item $g_1: x_1 \leftrightarrow -x_1, x_2 \leftrightarrow -x_2, z_2 \leftrightarrow -z_2, z_3 \leftrightarrow -z_3$
\item $g_2: y_1 \leftrightarrow -y_1, y_2 \leftrightarrow -y_2, z_2 \leftrightarrow -z_2, z_3 \leftrightarrow -z_3$
\item $g_3: x_1 \leftrightarrow -x_1, y_2 \leftrightarrow -y_2, z_1 \leftrightarrow -z_1, z_2 \leftrightarrow -z_2$
\item $g_1g_2: x_1 \leftrightarrow -x_1, x_2 \leftrightarrow -x_2, y_1 \leftrightarrow -y_1, y_2 \leftrightarrow -y_2$
\item $g_1g_3: x_2 \leftrightarrow -x_2, y_2 \leftrightarrow -y_2, z_1 \leftrightarrow -z_1, z_3 \leftrightarrow -z_3$
\item $g_2g_3: x_1 \leftrightarrow -x_1, y_1 \leftrightarrow -y_1, z_1 \leftrightarrow -z_1, z_3 \leftrightarrow -z_3$
\item $g_1g_2g_3: x_2 \leftrightarrow -x_2, y_1 \leftrightarrow -y_1, z_1 \leftrightarrow -z_1, z_2 \leftrightarrow -z_2$
\end{enumerate}
\end{theorem}

\begin{proof}
To prove (1), we show that $g_1$ can be realized by conjugating the matrix $G_n$ by a 
diagonal matrix $P_n$ of all diagonal entries equal to $\pm 1$. The diagonal entries of $P_n$ are defined as
$$p_{ii} = (-1)^{q(b_i,F_x)}$$
where $q(b_i,F_x)$ is the number of times $b_i$ intersects $F_x$ modulo $2$, cf. Figure \ref{fig8}. The theorem follows 
because curves corresponding to the variables $x_1$, $x_2$, $z_2$ and $z_3$ 
intersect $F_x\cup {F_x}^*$ in an odd number of points, whereas curves 
corresponding to the variables $d$, $z_2$, $y_1$ and $y_2$ cut 
it an even number of times.

\begin{figure}[htbp]
\begin{center}
\includegraphics[scale=1]{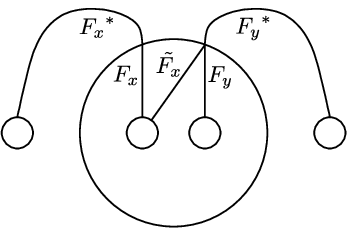} 
\caption{}
\label{fig8}
\end{center}
\end{figure}

More precisely, for
$$g_{ij}=\langle b_i,b_j \rangle = d^{n_d}{x_1}^{n_{x_1}}{x_2}^{n_{x_2}}{y_1}^{n_{y_1}}{y_2}^{n_{y_2}}{z_1}^{n_{z_1}}{z_2}^{n_{z_2}}{z_3}^{n_{z_3}},$$
the entry $g'_{ij}$ of $P_nG_n{P_n}^{-1}$ satisfies:
\begin{eqnarray*}
g'_{ij} &= &p_{ii}g_{ij}p_{jj}
\\&= &p_{ii}p_{jj}g_{ij}
\\&= &(-1)^{q(b_i,F_x)+ q(b_j,F_x)}g_{ij}
\\&= &(-1)^{n_{x_1}+n_{x_2} + n_{z_2} + n_{z_3}}g_{ij}
\\&= &d^{n_d}{(-x_1)}^{n_{x_1}}{(-x_2)}^{n_{x_2}}{y_1}^{n_{y_1}}{y_2}^{n_{y_2}}{z_1}^{n_{z_1}}{(-z_2)}^{n_{z_2}}{(-z_3)}^{n_{z_3}}
\end{eqnarray*}

For (2) and (3), we use the same method of proof as for (1). In (2), we use $F_y$ and $F_y\cup {F_y}^*$. In (3), we use $\tilde F_x$ and $\tilde F_x\cup  {F_y}^*$. (4) through (7) follow directly from (1), (2) and (3).
\end{proof}

\section{Terms of Maximal Degree in $\det G_n$}
Theorem 3.1 proves that the product of the diagonal entries of $G_n$ is the unique term of maximal degree, $n(n+1){\binom{2n}n}$, in $\det G_n$ consisting only of the variables $d$ and $z_1$. More precisely, the product of the diagonal of $G_n$ is given by
$$\delta(n) = \prod_{b_i \in \textbf{B}_n} \langle b_i, b_i \rangle = d^{\alpha(n)}z_1^{\beta(n)}$$ with $\alpha(n)+\beta(n) = n(n+1){\binom{2n}n}$.
$\delta(n)$ for the first few $n$ are given here:
$$\delta(1) = d^2z_1^2 \qquad \delta(2) = d^{20}z_1^{16} \qquad \delta(3)=d^{144}z_1^{96} \qquad \delta(4)=d^{888}z_1^{512}$$

Computing the general formula for $\delta(n)$ can be reduced to a purely 
combinatorial problem. We conjectured that $\beta(n) = (2n)4^{n-1}$ 
and this was in fact proven by Louis Shapiro using an involved generating 
function argument \cite{Sha}. The result is stated formally below.\\

\begin{theorem}
$$\delta(n)= d^{n(n+1){2n \choose n}-(2n)4^{n-1}}z_1^{(2n)4^{n-1}}$$
\end{theorem}

Let $h(\det G_n)$ denote the truncation of $\det G_n$ to terms of maximal degree, that is, of degree $n(n+1){\binom{2n}n}$. Each term is a product of $(n+1){\binom{2n}n}$ entries in $G_n$, each of which is a monomial of degree $n$. By Lemma 3.1, $\langle b_i, b_j \rangle$ has degree $n$ if and only if $b_i$ and $b_j$ have the same underlying Catalan state. There are $C_n = \frac{1}{n+1}{\binom{2n}n}$ elements in $\textbf{B}_{n,0}$. Divide $\textbf{B}_n$ into subsets corresponding to underlying Catalan states, that is, into subsets $A_1, \ldots, A_{C_n}$, such that for all $b_i, b_j \in A_k$, $\gamma(b_i) = \gamma(b_j)$. Then from Lemma 3.1 we have
\begin{proposition}
$$h(\det G_n) = \prod_{k=1}^{C_n}\det \langle A_k, A_k \rangle$$
\end{proposition}

Note that $\langle A_k, A_k \rangle$ are simply blocks in $G_n$ whose determinants can be multiplied together to give the highest terms in $\det G_n$. Finding the terms of maximal degree in $\det G_n$ can give insight into decomposition of $\det G_n$ for large $n$.

\begin{example}
$\textbf{B}_1$ corresponds to the single Catalan state in $\textbf{B}_{1,0}$. Thus, $\det G_1 = h(\det G_1)$, a homogeneous polynomial of degree 4.
\end{example}

\begin{example}
$\textbf{B}_2$ can be divided into two subsets, corresponding to the two Catalan states in $\textbf{B}_{2,0}$. We can thus find $h(\det G_2)$ by computing two $9 \times 9$ block determinants. The two Catalan states in $\textbf{B}_{2,0}$ are equivalent up to rotation, so the two block determinants are equal. Specifically, we have:
\begin{eqnarray*}
h(\det G_2) &= &d^6 (x_1 x_2+x_2 y_1+x_1 y_2+y_1 y_2-d z_1-z_1 z_2-d z_3-z_2 z_3)^4\\
&&(-x_1 x_2+x_2 y_1+x_1 y_2-y_1 y_2+d z_1-z_1 z_2-d z_3+z_2 z_3)^4\\
&&(-x_1 x_2 z_1-y_1 y_2 z_1+{d z_1}^2+x_2 y_1 z_3+x_1 y_2 z_3-{d z_3}^2)^2\\
&&(-2 x_1 x_2 y_1 y_2+d x_1 x_2 z_1+d y_1 y_2 z_1-d^2 {z_1}^2+d x_2 y_1
z_3+d x_1 y_2 z_3-d^2 {z_3}^2)^2\\
&= &d^6 {\det G_1}^4(-x_1 x_2 z_1-y_1 y_2 z_1+{d z_1}^2+x_2 y_1 z_3+x_1 y_2 z_3-{d z_3}^2)^2\\
&&(-2 x_1 x_2 y_1 y_2+d x_1 x_2 z_1+d y_1 y_2 z_1-d^2 {z_1}^2+d x_2 y_1
z_3+d x_1 y_2 z_3-d^2 {z_3}^2)^2
\end{eqnarray*}
\end{example}

\begin{example}
$\textbf{B}_3$ can be divided into five subsets, corresponding to the five Catalan states in $\textbf{B}_{3,0}$. We can thus find $h(\det G_3)$ by computing the determinants of five blocks in $\textbf{B}_3$. The determinant of each block gives a homogeneous polynomial of degree $240/5 = 48$. $\textbf{B}_{3,0}$ forms two equivalence classes up to rotation, so there are only two unique block determinants. For precise terms, we refer the reader to the Appendix.
\end{example}

\section{$\det G_{n-1}$ Divides $\det G_n$}
In this section, we prove that the Gram determinant for $n-1$ chords divides the Gram determinant for $n$ chords. We need several lemmas:

\begin{lemma}
For any $b_i \in \textbf{B}_n$, $p_0(b_i) \in \textbf{B}_{n-1}$ if and only if $b_i$ contains no chord connecting $a_0$ and $a_{2n-1}$.
\end{lemma}

\begin{proof}
Suppose $a_0$ and $a_{2n-1}$ are not connected by a chord in $b_i$, 
say, $a_0$ is connected to $a_j$ and $a_{2n-1}$ is connected to $a_k$.
 Then $p_0(b_i)$ connects $a_0$ and $a_{2n-1}$ by a chord outside the outer 
boundary, and this chord does not form a closed curve. Because $a_{j}$ is 
connected to $a_0$ and $a_{k}$ is connected to $a_{2n-1}$, $p_0(b_i)$ 
contains single path from $a_{k}$ to $a_{j}$, which we can deform through 
isotopy so that it fits inside the outer circle. 
Thus, $p_0(b_i) \in \textbf{B}_{n-1}$, cf. Figure \ref{fig9}.

\begin{figure}[htbp]
\begin{center}
\includegraphics{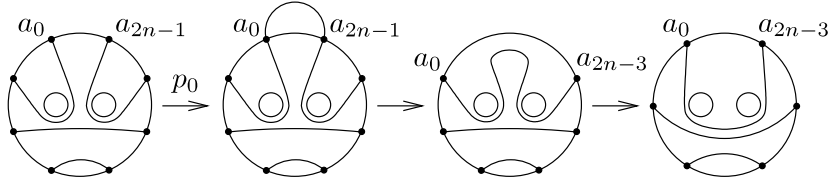}
\end{center}
\caption{From $b_i$, we obtain $p_0(b_i)$ by adjoining a chord outside the outer boundary between $a_0$ and $a_{2n-1}$, and pushing the chord inside the boundary. If $b_i$ does not contain a chord connecting $a_0$ and $a_{2n-1}$, then $p_0(b_i) \in \textbf{B}_{n-1}$.}
\label{fig9}
\end{figure}

If $b_i$ contains an arc connecting $a_0$ and $a_{2n-1}$, then $p_0(b_i)$ contains a closed curve enclosing some subset of $\{X_1, Y_1\}$, and cannot be in $\textbf{B}_n$.
\end{proof}

\begin{lemma}
For any $b_i \in \textbf{B}_n$, if $p_0(b_i) \notin \textbf{B}_{n-1}$, there exists $b_{\alpha(i)} \in \textbf{B}_{n-1}$ such that, for all $b_j \in \textbf{B}_{n-1}$, one of the following is true:
\begin{enumerate}
\item $\langle p_0(b_i), b_j\rangle = d\langle b_{\alpha(i)}, b_j \rangle$
\item $\langle p_0(b_i), b_j\rangle = x_1\langle b_{\alpha(i)}, b_j \rangle$
\item $\langle p_0(b_i), b_j\rangle = y_1\langle b_{\alpha(i)}, b_j \rangle$
\item $\langle p_0(b_i), b_j\rangle = z_2\langle b_{\alpha(i)}, b_j \rangle$.
\end{enumerate}
\end{lemma}

\begin{proof}
By Lemma 5.1, $b_i$ contains a chord connecting points $a_0$ and $a_{2n-1}$, so $p_0(b_i)$ must consist of some diagram in $\textbf{B}_{n-1}$ and a closed curve enclosing some subset of $\{X_1, Y_1\}$. The former is given by $\langle b_{\alpha(i)}, b_j \rangle$ for some $b_{\alpha(i)} \in \textbf{B}_{n-1}$, and the latter curve is given by one of the following variables: $d, x_1, y_1, z_2$.
\end{proof}

The previous two lemmas, combined with Proposition 2.1, leads to the following corollary.
\begin{corollary}
Let $\mathcal{A} = \{1, d, x_1, y_1, z_2\}$. For any $b_i \in \textbf{B}_n$, there exists $b_{\alpha(i)} \in \textbf{B}_{n-1}$ and $c \in \mathcal{A}$ such that $\langle b_i, i_0(\textbf{B}_{n-1}) \rangle = c\langle b_{\alpha(i)}, i_0(\textbf{B}_{n-1}) \rangle$.
\end{corollary}
That is, the rows of $\langle \textbf{B}_n, i_0(\textbf{B}_{n-1}) \rangle$ are each either equal to some row of $G_{n-1}$, or to some row of $G_{n-1}$ multiplied by one of the following variables: $d$, $x_1$, $y_1$, $z_2$. We now have all the lemmas needed for our main result of this section.
\begin{theorem}
For $n>1$, $\det G_{n-1}|\det G_n$.
\end{theorem}

\begin{proof}
We begin by proving that for every row of the matrix $G_{n-1}$, there exists 
an equivalent row in the submatrix $\langle \textbf{B}_n, i_0(\textbf{B}_{n-1})\rangle$ of $G_n$. Fix $b_i \in \textbf{B}_{n-1}$ and take the row of $G_{n-1}$ given by $\langle b_i, \textbf{B}_{n-1}\rangle$. We claim that the row in $\langle \textbf{B}_n, i_0(\textbf{B}_{n-1})\rangle$ given by $\langle i_1(b_i), i_0(\textbf{B}_{n-1}) \rangle$ is equal to $\langle b_i, \textbf{B}_{n-1}\rangle$. In other words, $\langle i_1(b_i), i_0(\textbf{B}_{n-1})\rangle$ is equal to the $i^\text{th}$ row of $G_{n-1}$, a fact which we leave to the reader for the moment, but will demonstrate explicitly in the next section, cf. Theorem 6.1.\\

Reorder the elements of $\textbf{B}_n$ so that $\langle i_0(\textbf{B}_{n-1}),i_0(\textbf{B}_{n-1})\rangle$ forms an upper-leftmost block of $G_n$ and $\langle i_1(\textbf{B}_{n-1}),i_0(\textbf{B}_{n-1})\rangle$ forms a block directly underneath $\langle i_0(\textbf{B}_{n-1}),i_0(\textbf{B}_{n-1})\rangle$. This is illustrated below:
\begin{eqnarray*}
G_n &= &\begin{pmatrix}
\langle i_0(\textbf{B}_{n-1}),i_0(\textbf{B}_{n-1})\rangle & *&*&*&*&*\\
\langle i_1(\textbf{B}_{n-1}),i_0(\textbf{B}_{n-1})\rangle & *&*&*&*&*\\
*&*&*&*&*&*\\
*&*&*&*&*&*\\
*&*&*&*&*&*\\
*&*&*&*&*&*\\
\end{pmatrix}
\\&=&\begin{pmatrix}
\langle i_0(\textbf{B}_{n-1}),i_0(\textbf{B}_{n-1})\rangle & *&*&*&*&*\\
G_{n-1} & *&*&*&*&*\\
*&*&*&*&*&*\\
*&*&*&*&*&*\\
*&*&*&*&*&*\\
*&*&*&*&*&*\\
\end{pmatrix}
\end{eqnarray*}

Corollary 5.1 implies that every row of $\langle \textbf{B}_n, i_0(\textbf{B}_{n-1})\rangle$ is a multiple of some row in $G_{n-1}$. Let $j_1, \ldots, j_k$ denote the indices of all rows of $\langle \textbf{B}_n, i_0(\textbf{B}_{n-1})\rangle$ other than those in $\langle i_1(\textbf{B}_{n-1}),i_0(\textbf{B}_{n-1})\rangle$. Let ${G_n}'$ be the matrix obtained by properly subtracting multiples of rows in $\langle i_1(\textbf{B}_{n-1}),i_0(\textbf{B}_{n-1})\rangle$ from rows $j_1, \ldots, j_k$ of $G_n$ so that the submatrix obtained by restricting ${G_n}'$ to rows $j_1, \ldots, j_k$ and columns corresponding to states in $i_0(\textbf{B}_{n-1})$ is equal to $0$:
\begin{eqnarray*}
{G_n}' &= &\begin{pmatrix}
0 & *&*&*&*&*\\
G_{n-1} & *&*&*&*&*\\
0 & *&*&*&*&*\\
0&*&*&*&*&*\\
0&*&*&*&*&*\\
0&*&*&*&*&*\\
\end{pmatrix}
\end{eqnarray*}
Thus, ${G_n}'$ restricted to the columns corresponding to states in $i_0(\textbf{B}_{n-1})$ contains precisely $n{\binom{2n-2}{n-1}}$ nonzero rows, each equal to some unique row of $G_{n-1}$. The determinant of this submatrix is equal to $\det G_{n-1}$. Since $\det G_{n-1} | \det {G_n}'$ and $\det {G_n}' = \det G_n$, this completes the proof.
\end{proof}

\section{Further Relation Between $\det G_{n-1}$ and $\det G_n$}
As was first noted in the proof of Theorem 5.1, there exists a submatrix of 
$G_n$ equal to $G_{n-1}$. This section will be focused on identifying 
multiple nonoverlapping submatrices in $G_n$ equal to multiples of $G_{n-1}$. 
This will prove useful for simplifying the computation of $\det G_n$. 
We start with the main lemma for this section and for Theorem 5.1.

\begin{lemma}
For any $b_i,b_j \in \textbf{B}_{n-1}$, $\langle i_0(b_i), i_1(b_j) \rangle = \langle i_1(b_i), i_0(b_j) \rangle = \langle b_i, b_j \rangle$.
\end{lemma}
\begin{proof}
We begin with the equality $\langle i_1(b_i), i_0(b_j) \rangle = \langle b_i, b_j \rangle$. By Proposition 2.1, $i_1(b_i) \circ {i_0(b_j)}^* = p_0i_1(b_i) \circ {b_j}^*$, so it suffices to prove that $p_0i_1(b_i) = p_0r_{\pi/n}i_0{r_{\pi/n-1}}^{-1}(b_i) = b_i$. This is demonstrated pictorially:
\begin{figure}[htbp]
\begin{center}
\includegraphics[scale=1]{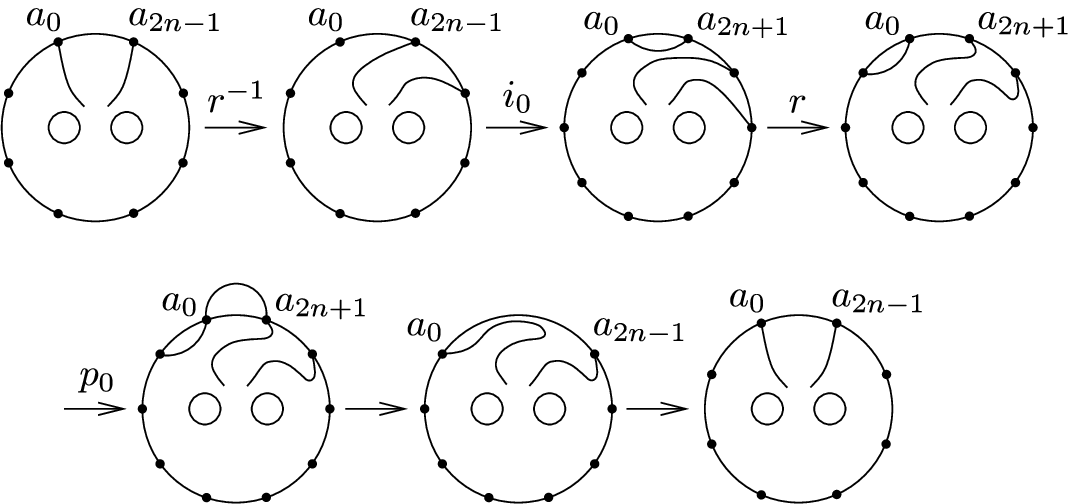}
\caption{}
\label{fig10}
\end{center}
\end{figure}

Thus, $\langle i_1(b_i), i_0(b_j) \rangle = \langle b_i, b_j \rangle$. Recall that $\langle b_i, b_j \rangle = h_t(\langle b_j, b_i \rangle)$. From this and the previous equality, it follows that
$$\langle i_0(b_i), i_1(b_j) \rangle = h_t(\langle i_1(b_j), i_0(b_i) \rangle)= h_t(\langle b_j, b_i \rangle) = {h_t}^2(\langle b_i, b_j \rangle) = \langle b_i, b_j \rangle.$$
\end{proof}

\begin{corollary}
$\langle i_0(\textbf{B}_{n-1}), i_1(\textbf{B}_{n-1})\rangle = \langle i_1(\textbf{B}_{n-1}), i_0(\textbf{B}_{n-1}) \rangle = G_{n-1}$.
\end{corollary}

\begin{lemma}
For any $b_i, b_j \in \textbf{B}_{n-1}$, $\langle i_0(b_i),i_0(b_j)\rangle = \langle i_1(b_i), i_1(b_j)\rangle = d\langle b_i,b_j\rangle$.
\end{lemma}
\begin{proof}
$i_0(b_i) \circ {i_0(b_j)}^*$ is composed of $b_i \circ {b_j}^*$ in addition to a chord close to the boundary glued with its inverse. The latter pairing gives a trivial circle. Thus, $\langle i_0(b_i), i_0(b_j)\rangle = d\langle b_i, b_j\rangle$ for all $b_i, b_j \in \textbf{B}_{n-1}$.\\

By symmetry, $\langle i_1(\textbf{B}_{n-1}), i_1(\textbf{B}_{n-1})\rangle = dG_{n-1}$.
\end{proof}
\begin{corollary}
$\langle i_0(\textbf{B}_{n-1}), i_0(\textbf{B}_{n-1})\rangle = \langle i_1(\textbf{B}_{n-1}), i_1(\textbf{B}_{n-1})\rangle = dG_{n-1}$.
\end{corollary}

Using these two facts, we can construct from $G_n$ a $(|B_n| - 2|B_{n-1}|) 
\times (|B_n| - 2|B_{n-1}|)$ matrix whose determinant is equal to $\det G_n / (1-d^2)^{n{\binom{2n-2}{n-1}}}{\det G_{n-1}}^2$. This allows us to compute $\det G_n$ with greater ease, assuming we know $\det G_{n-1}$. This process is shown in the next theorem.

\begin{theorem}
There exists an integer $k \ge 0$\footnote{Clearly $k$ is bounded above by 
$(n+1){\binom{2n}n}$, or even better, by $|\textbf{B}_n| - 2|\textbf{B}_{n-1}|$. There are obviously better approximations possible, but we do not address them in this paper.} such that, for all integers $n > 1$, $${\det G_{n-1}}^2 | \det G_n (1-d^2)^k.$$
\end{theorem}

\begin{proof}
Order the elements of $\textbf{B}_n$, (or equivalently, the rows and columns of $G_n$) as shown in Theorem 5.1.
\iffalse
\begin{eqnarray*}
G_n &= &\begin{pmatrix}
\langle i_0(\textbf{B}_{n-1}),i_0(\textbf{B}_{n-1})\rangle & \langle i_0(\textbf{B}_{n-1}),i_1(\textbf{B}_{n-1})\rangle & *&*&*&*\\
\langle i_1(\textbf{B}_{n-1}),i_0(\textbf{B}_{n-1})\rangle & \langle i_1(\textbf{B}_{n-1}),i_1(\textbf{B}_{n-1})\rangle & *&*&*&*\\
* & * & *&*&*&*\\
* & * & *&*&*&*\\
* & * & *&*&*&*\\
* & * & *&*&*&*
\end{pmatrix}
\\&=&\begin{pmatrix}
dG_{n-1} & G_{n-1} & * & *&*&*\\
G_{n-1} & dG_{n-1} & *&*&*&*\\
* & * & *&*&*&*\\
* & * & *&*&*&*\\
* & * & *&*&*&*\\
* & * & *&*&*&*
\end{pmatrix}
\end{eqnarray*}
\fi
We apply the procedure from Theorem 5.1 to construct ${G_n}'$, whose form is given roughly below:
\begin{eqnarray*}
{G_n}' &= &
\begin{pmatrix}
0 & (1-d^2)G_{n-1} & \begin{matrix}*&*&*&*\end{matrix}\\
G_{n-1} & dG_{n-1} & \begin{matrix}*&*&*&*\end{matrix}\\
\begin{matrix}0\\0\\0\\0\end{matrix}&\framebox{$\begin{matrix}*\\ *\\ *\\ *\end{matrix}$}& \begin{matrix}*&*&*&*\\ *&*&*&*\\ *&*&*&*\\ *&*&*&*\end{matrix}\\
\end{pmatrix}
\end{eqnarray*}
Consider the block in ${G_n}'$ whose columns correspond to states in $i_1(\textbf{B}_{n-1})$ and whose rows correspond to states in neither $i_0(\textbf{B}_{n-1})$ nor $i_1(\textbf{B}_{n-1})$ (boxed above). Every row in this submatrix is a linear combination of two rows from $G_{n-1}$. More precisely, each row is of the form $a_1l_1-a_2dl_2$, where $l_1$ and $l_2$ are two rows, not necessarily distinct, in $G_{n-1}$, and $a_1, a_2 \in \mathcal{A} = \{1,d, x_1, y_1, z_2\}$. If we assume $(1-d^2)$ is invertible in our ring, then each row is a linear combination of two rows from $(1-d^2)G_{n-1}$. We then simplify ${G_n}'$ as follows.\\

Let ${G_n}''$ be the matrix obtained by properly subtracting linear combinations of the first $n{\binom{2n-2}{n-1}}$ rows of ${G_n}'$ from the rows which correspond to states in neither $i_0(\textbf{B}_{n-1})$ nor $i_1(\textbf{B}_{n-1})$ so that the submatrix obtained by restricting ${G_n}''$ to columns corresponding to states in $i_1(\textbf{B}_{n-1})$ and rows corresponding to states in neither $i_0(\textbf{B}_{n-1})$ nor $i_1(\textbf{B}_{n-1})$ is equal to $0$:
\begin{eqnarray*}
{G_n}'' &= &\begin{pmatrix}
0 & (1-d^2)G_{n-1} &\begin{matrix}*&*&*&*\end{matrix}\\
G_{n-1} & dG_{n-1} &\begin{matrix}*&*&*&*\end{matrix}\\
\begin{matrix}0\\0\\0\\0\end{matrix}&\begin{matrix}0\\0\\0\\0\end{matrix}&
\framebox{$\begin{matrix}*&*&*&*\\ *&*&*&*\\ *&*&*&*\\ *&*&*&*\end{matrix}$}
\end{pmatrix}
\end{eqnarray*}
The block decomposition at this point proves that $\det {G_n}''$ is equal to $(1-d^2)^{n{\binom{2n-2}{n-1}}}({\det G_{n-1}})^2$ times the determinant of the boxed block, which we denote by $\bar{G_n}$. The latter contains a power of $(1-d^2)^{-1}$, whose degree is unspecified. Thus, ${\det G_{n-1}}^2 | \det {G_n}'' (1-d^2)^k$ for some integer $k \ge 0$.
We remind the reader that ${G_n}''$ is obtained from ${G_n}'$ via determinant preserving operations, and hence $\det {G_n}' = \det G_n$.
\end{proof}

Note that if $\det \bar{G_n}$ has fewer than $n{\binom{2n-2}{n-1}}$ powers of $(1-d^2)^{-1}$, then ${\det G_{n-1}}^2 | \det {G_n}$. It remains an open problem as to whether this is true. For an example of this decomposition, we refer the reader to the Appendix.

\section{Future Directions}
In this section, we discuss briefly generalizations of the Gram determinant and present a number of open questions and conjectures.

\subsection{The case of a disk with $k$ holes}
We can generalize our setup by considering $F_{0,k}^n$, a unit disk with $k$ holes, 
in addition to $2n$ points, $a_0, \ldots, a_{2n-1}$, arranged in a similar 
way to points in $F_{0,2}^n$. For $b_i, b_j \in \textbf{B}_{n,k}$, 
let $b_i \circ {b_j}^*$ be defined in the same way as before. 
Each paired diagram $b_i \circ {b_j}^*$ consists of up to $n$ closed curves 
on the $2$-sphere ($D^2 \cup {(D^2)}^*$) with $2k$ holes. Let $\mathcal{S}$ denote the 
set of all $2k$ holes. We differentiate between the closed curves based on 
how they partition $\mathcal{S}$. We define a bilinear form by counting 
the multiplicities of each type of closed curve in the paired diagram. 
In the case $k=2$, we assigned to each paired diagram a corresponding 
element in a polynomial ring of eight variables, each variable representing 
a type of closed curve. In the general case, the number of types of closed 
curves is equal to
$$\frac{2^{|S|}}{2} = \frac{2^{2k}}{2} = 2^{2k-1}$$
so we can define the Gram matrix of the bilinear form for a disk with $k$ 
holes and $2n$ points with $(n+1)^{k-1}{\binom{2n}n} \times 
(n+1)^{k-1}{\binom{2n}n}$ entries, each belonging to a polynomial ring of 
$2^{2k-1}$ variables. We denote this Gram matrix by $G_n^{F_{0,k}}$. 
For $n=1$ and $k=3$, we can easily write this $8\times 8$ Gram matrix. For purposes of notation, let us denote the holes in 
$F^n_{0,3}$ by $\partial_1$, $\partial_2$ and $\partial_3$, and 
their inversions by $\partial_{-1}$, $\partial_{-2}$ and $\partial_{-3}$, 
respectively. Hence, each closed curve in the surface encloses some subset 
of $\mathcal{S} = \{\partial_1,\partial_{-1},\partial_2,\partial_{-2},\partial_3,
\partial_{-3}\}$. Let $x_{a_1,a_2,a_3}$ denote a curve separating the set of 
holes $\{\partial_{a_1}, \partial_{a_2}, \partial_{a_3}\}$ from $\mathcal{S} - 
\{\partial_{a_1}, \partial_{a_2}, \partial_{a_3}\}$. 
We can similarly define $x_{a_1, a_2}$ and $x_{a_1}$. 
The Gram matrix is then:
$$ G_1^{F_{0,3}}=
\begin{pmatrix}
d & x_{-3} & x_{-2} & x_{-2,-3} & x_{-1} & x_{-1,-3} & x_{-1,-2} & x_{1,2,3}\\
x_3 & x_{3,-3} & x_{-2, 3} & x_{1,-1,2} & x_{-1, 3} & x_{1,2,-2} & x_{1,2,-3} & x_{1,2}\\
x_2 & x_{2,-3} & x_{2,-2} & x_{1,-1,3} & x_{-1,2} & x_{1,-2,3} & x_{1,3,-3} & x_{1,3}\\
x_{2,3} & x_{1,-1,-2} & x_{1,-1,-3} & x_{1,-1} & x_{1,-2,-3} & x_{1,-2} & x_{1,-3} & x_1\\
x_1 & x_{1,-3} & x_{1,-2} & x_{1,-2,-3} & x_{1,-1} & x_{1,-1,-3} & x_{1,-1,-2} & x_{2,3}\\
x_{1,3} & x_{1,3,-3} & x_{1,-2,3} & x_{-1,2} & x_{1,-1,3} & x_{2,-2} & x_{2,-3} & x_2\\
x_{1,2} & x_{1,2,-3} & x_{1,2,-2} & x_{1,-3} & x_{1,-1,2} & x_{-2,3} & x_{3,-3} & x_3\\
x_{1,2,3} & x_{-1,-2} & x_{-1,-3} & x_{-1} & x_{-2,-3} & x_{-2} & x_{-3} & d
\end{pmatrix}
$$
It would be tempting to conjecture that the determinant of the above matrix 
has a straightforward decomposition of the form $(u+v)(u-v)$. We found that 
it is the case for the substitution $x_{a_1} = x_{a_1,a_2} = 0$ with $a_1, a_2 \in \{-3,-2,-1,1,2,3\}$ (see Appendix). However in general, 
the preliminary calculation suggests that $\det G_n^{F_{0,3}}$ may be 
an irreducible polynomial.\\
%We also predict the general form of 
%$\det G_1^{F_{0,k}}$ for $k>0$ to be $(u+v)(u-v)$.

Finally, we observe that many results we have proven for 
$\det G_n^{F_{0,2}}$ holds for general $\det G_n^{F_{0,k}}$. For example, 
$\det G_n^{F_{0,k}}\neq 0$ and $\det G_n^{F_{0,k-1}}|\det G_n^{F_{0,k}}$.
 In the specific case of $\det G_n^{F_{0,3}}$ we conjecture that the 
diagonal term is of the form 
$\delta(n) =d^{\alpha(n)}(x_{1,-1}x_{2,-2}x_{3,-3})^{\beta(n)}$, where 
$\alpha(n)+3\beta(n)=n(n+1)^2{2n\choose n}$ and $\beta(n)= n(n+1)4^{n-1}$.

\subsection{Speculation on factorization of $\det G_n$}

Section 5 establishes that $\det G_{n-1} | \det G_n$, but we conjecture 
that there are many more powers of $\det G_{n-1}$ in $\det G_n$. 
Indeed, even in the base case, ${\det G_1}^k|\det G_2$ for $k$ up to $4$. 
Finding the maximal power of $\det G_{n-1}$ in $\det G_n$ in the general case 
is an open problem and can be helpful toward computing the full decomposition 
of $\det G_n$.\\

Examining the terms of highest degree in $\det G_n$, that is, $h(\det G_n)$ 
may also yield helpful hints toward the full decomposition. 
In particular, we note that:
$${\det G_1}^4 | h(\det G_2) \quad \text{and} \quad 
\left(\frac{{h(\det G_2)}^6}{{\det G_1}^9}\right) | h(\det G_3)$$

We can conjecture that
$$\left(\frac{{\det G_2}^6}{{\det G_1}^9}\right)|\det G_3$$
so it follows that ${\det G_1}^{15} | \det G_3$. 
We therefore offer the following conjecture:

\begin{conjecture}
${\det G_1}^{\binom{2n}{n-1}} | \det G_n$ for $n \ge 1$.
\end{conjecture}

In addition, we also offer the following conjecture, motivated by observations 
of $\det G_1$ and $\det G_2$:

\begin{conjecture}
Let $H_n$ denote the factors of $\det G_n$ not in $\det G_{n-1}$, that is, $H_n | \det G_n$ and $\gcd(H_n, \det G_{n-1}) = 0$. Then $(H_{n-1})^{2n} | \det G_n$.
\end{conjecture}

\begin{conjecture}\label{Conjecture 3}
Let, as before, $R= \mathbb{Z} [d,x_1,x_2,y_1,y_2,z_1,z_2,z_3]$ and 
$R_1$ be a subgroup of $R$ of elements invariant under $h_1,h_2,h_t$, and 
$g_1,g_2,g_3$. Similarly, let $R_2$ be a subgroup of $R$ composed of 
elements $w\in R$ such that $h_1(w) = h_2(w) = -w$ and 
$h_t(w)=g_1(w)=g_2(w)=g_3(w)$. Then
\begin{enumerate}
\item[(1)] $\det G_n = u^2 - v^2$, where $u\in R_1$ and $v\in R_2$.
\item[(2)] $\det G_n = \prod_{\alpha} (u_{\alpha}^2 - v_{\alpha}^2)$, 
where $u_{\alpha} \in R_1$ and $v_{\alpha} \in R_2$, and 
$u_{\alpha} - v_{\alpha}$ and $u_{\alpha} + v_{\alpha}$ are 
irreducible polynomials.
\item[(3)] $\det G_n = \prod_{i=1}^n ({u_i}^2-{v_i}^2)^{\binom{2n}{n-i}}$, 
where $u_i \in R_1$ and $v_i \in R_2$.
\end{enumerate}
Notice that if $w_1=u_1^2-v_1^2$ and $w_2=u_2^2-v_2^2$, then 
$w_1w_2 = (u_1u_2 + v_1v_2)^2- (u_1v_2+ u_2v_1)^2$.
\end{conjecture}

We have little confidence in Conjecture 3(3). 
It is closely, maybe too closely, influenced by the Gram determinant of 
type B ($\det G_n^B= \det G_n^{F_{0,1}}$).
That is
\begin{theorem}(\cite{Che,MS})
$$\det G_n^B= \prod_{i=1}^n \left( T_i(d)^2 - a^2 \right)^{\binom{2n}{n-i}}
$$
where $T_i(d)$ is the Chebyshev polynomial of the first kind:
$$
T_0 = 2, \qquad T_1 = d, \qquad T_i = d\, T_{i-1} - T_{i-2};
$$
$d$ and $a$ in the formula, correspond to the trivial and the nontrivial 
curves in the annulus $F_{0,1}$, respectively.
\end{theorem}

\newpage
\begin{landscape}
\section{Appendix}
\subsection{}$\textbf{B}_2$
\begin{figure}[htbp]
\begin{center}
\includegraphics{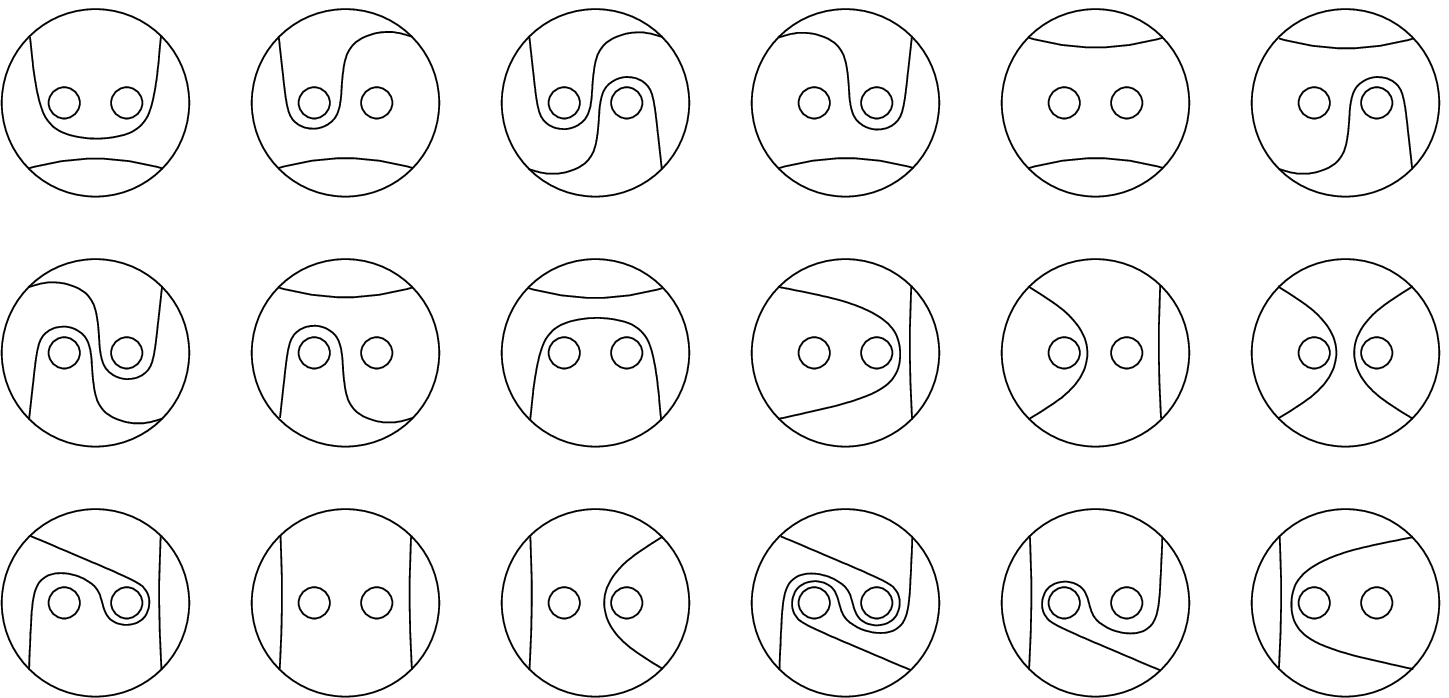}
\end{center}
\end{figure}

\subsection{}$G_2$
\begin{tiny}
\begin{eqnarray*}
\left(
\begin{array}{cccccccccccccccccc}
 d^2 & d y_2 & d x_2 & d z_2 & z_2 & x_2 & y_2 & d & y_2 z_2 & x_2 z_2 & {z_2}^2
& z_2 & x_2 & y_2 & {y_2}^2 & {x_2}^2 & z_2 & z_2 \\
 d y_1 & d z_1 & d z_3 & d x_1 & x_1 & z_3 & z_1 & y_1 & x_1 y_2 & x_1 x_2
& x_1 z_2 & x_1 & z_3 & z_1 & y_2 z_1 & x_2 z_3 & x_1 & x_1 \\
 d x_1 & d z_3 & d z_1 & d y_1 & y_1 & z_1 & z_3 & x_1 & y_1 y_2 & x_2 y_1
& y_1 z_2 & y_1 & z_1 & z_3 & y_2 z_3 & x_2 z_1 & y_1 & y_1 \\
 d z_2 & d x_2 & d y_2 & d^2 & d & y_2 & x_2 & z_2 & d y_2 & d x_2 & d z_2 & d & y_2
& x_2 & x_2 y_2 & x_2 y_2 & d & d \\
 z_2 & x_2 & y_2 & d & d^2 & d y_2 & d x_2 & d z_2 & y_2 & x_2 & z_2 & {z_2}^2 & x_2
z_2 & y_2 z_2 & z_2 & z_2 & {x_2}^2 & {y_2}^2 \\
 x_1 & z_3 & z_1 & y_1 & d y_1 & d z_1 & d z_3 & d x_1 & z_1 & z_3 & x_1 & x_1
z_2 & x_1 x_2 & x_1 y_2 & x_1 & x_1 & x_2 z_3 & y_2 z_1 \\
 y_1 & z_1 & z_3 & x_1 & d x_1 & d z_3 & d z_1 & d y_1 & z_3 & z_1 & y_1 & y_1
z_2 & x_2 y_1 & y_1 y_2 & y_1 & y_1 & x_2 z_1 & y_2 z_3 \\
 d & y_2 & x_2 & z_2 & d z_2 & d x_2 & d y_2 & d^2 & x_2 & y_2 & d & d z_2 & d x_2 &
d y_2 & d & d & x_2 y_2 & x_2 y_2 \\
 y_1 z_2 & x_2 y_1 & y_1 y_2 & d y_1 & y_1 & z_1 & z_3 & x_1 & d z_1 & d
z_3 & d x_1 & y_1 & z_1 & z_3 & x_2 z_1 & y_2 z_3 & y_1 & y_1 \\
 x_1 z_2 & x_1 x_2 & x_1 y_2 & d x_1 & x_1 & z_3 & z_1 & y_1 & d z_3 & d
z_1 & d y_1 & x_1 & z_3 & z_1 & x_2 z_3 & y_2 z_1 & x_1 & x_1 \\
 {z_2}^2 & x_2 z_2 & y_2 z_2 & d z_2 & z_2 & x_2 & y_2 & d & d x_2 & d y_2 & d^2
& z_2 & x_2 & y_2 & {x_2}^2 & {y_2}^2 & z_2 & z_2 \\
 z_2 & x_2 & y_2 & d & {z_2}^2 & x_2 z_2 & y_2 z_2 & d z_2 & y_2 & x_2 & z_2
& d^2 & d y_2 & d x_2 & z_2 & z_2 & {y_2}^2 & {x_2}^2 \\
 x_1 & z_3 & z_1 & y_1 & x_1 z_2 & x_1 x_2 & x_1 y_2 & d x_1 & z_1 & z_3
& x_1 & d y_1 & d z_1 & d z_3 & x_1 & x_1 & y_2 z_1 & x_2 z_3 \\
 y_1 & z_1 & z_3 & x_1 & y_1 z_2 & x_2 y_1 & y_1 y_2 & d y_1 & z_3 & z_1
& y_1 & d x_1 & d z_3 & d z_1 & y_1 & y_1 & y_2 z_3 & x_2 z_1 \\
 {y_1}^2 & y_1 z_1 & y_1 z_3 & x_1 y_1 & z_2 & x_2 & y_2 & d & x_1 z_1 & x_1
z_3 & {x_1}^2 & z_2 & x_2 & y_2 & {z_1}^2 & {z_3}^2 & z_2 & z_2 \\
 {x_1}^2 & x_1 z_3 & x_1 z_1 & x_1 y_1 & z_2 & x_2 & y_2 & d & y_1 z_3 & y_1
z_1 & {y_1}^2 & z_2 & x_2 & y_2 & {z_3}^2 & {z_1}^2 & z_2 & z_2 \\
 z_2 & x_2 & y_2 & d & {x_1}^2 & x_1 z_3 & x_1 z_1 & x_1 y_1 & y_2 & x_2 &
z_2 & {y_1}^2 & y_1 z_1 & y_1 z_3 & z_2 & z_2 & {z_1}^2 & {z_3}^2 \\
 z_2 & x_2 & y_2 & d & {y_1}^2 & y_1 z_1 & y_1 z_3 & x_1 y_1 & y_2 & x_2 &
z_2 & {x_1}^2 & x_1 z_3 & x_1 z_1 & z_2 & z_2 & {z_3}^2 & {z_1}^2
\end{array}
\right)
\end{eqnarray*}
\end{tiny}

\newpage
\subsection{}$\bar{G_2}$ (defined in Theorem 6.1), after simplification
\begin{tiny}
\begin{eqnarray*}
\left(
\begin{array}{cccccccccc}
 x_2 y_2-d z_2 & -(-1+d) d (1+d) & -d (-1+y_2) (1+y_2) & -d (-1+z_2) (1+z_2) & -(-1+d) (1+d) y_2
\\
 0 & -(d-y_1) (d+y_1) & -(y_2-z_1) (y_2+z_1) & (x_1-z_2) (x_1+z_2) & -d y_2+y_1
z_1\\
 0 & -(d-z_2) (d+z_2) & (x_2-y_2) (x_2+y_2) & (d-z_2) (d+z_2) & -d y_2+x_2 z_2
\\
 -d x_1+x_2 z_1-y_1 z_2+y_2 z_3 & -2 (-1+d) (1+d) y_1 & 2 y_1-y_1 {y_2}^2-d y_2
z_1 & 2 y_1-d x_1 z_2-y_1 {z_2}^2 & -d y_1 y_2+2 z_1-d^2 z_1  \\
 0 & -d y_1+x_1 z_2 & -y_2 z_1+x_2 z_3 & d y_1-x_1 z_2 & x_1 x_2-d z_1
\\
 -{x_1}^2-{y_1}^2+{z_1}^2+{z_3}^2 & -2 (d x_1 y_1-z_2) & -x_1 y_2 z_1+2 z_2-y_1 y_2
z_3 & -(-2+{x_1}^2+{y_1}^2) z_2 & 2 x_2-d x_1 z_1-d y_1 z_3 \\
 -d y_1+y_2 z_1-x_1 z_2+x_2 z_3 & -2 (-1+d) (1+d) x_1 & 2 x_1-x_1 {y_2}^2-d y_2
z_3 & 2 x_1-d y_1 z_2-x_1 {z_2}^2 & -d x_1 y_2+2 z_3-d^2 z_3\\
 0 & -d x_1+y_1 z_2 & x_2 z_1-y_2 z_3 & d x_1-y_1 z_2 & x_2 y_1-d z_3\\
 -d^2+{x_2}^2+{y_2}^2-{z_2}^2 & -2 (-1+d) (1+d) z_2 & -d x_2 y_2+2 z_2-{y_2}^2 z_2 & -z_2 (-2+d^2+{z_2}^2)
& 2 x_2-d^2 x_2-d y_2 z_2\\
 0 & -(d-x_1) (d+x_1) & -(y_2-z_3) (y_2+z_3) & (y_1-z_2) (y_1+z_2) & -d y_2+x_1
z_3
\end{array}
\right.
\\
\left.
\begin{array}{cccccccccc}
 y_2-d x_2 z_2 & -(-1+d) (1+d) x_2 & x_2-d y_2 z_2 & -(-1+d) (1+d) z_2 & -d (-1+x_2) (1+x_2)
\\
-x_2 z_2+x_1 z_3 & -d x_2+y_1 z_3 & x_1 z_1-y_2 z_2 & x_1 y_1-d
z_2 & -(x_2-z_3) (x_2+z_3) \\
 d y_2-x_2 z_2 & -d x_2+y_2 z_2 & d x_2-y_2 z_2 & 0 & -(x_2-y_2) (x_2+y_2)
\\
  -d x_1 x_2+2
z_1-x_2 y_1 z_2 & -d x_2 y_1+2 z_3-d^2 z_3 & -d x_1 y_2-y_1 y_2 z_2+2
z_3 & 2 x_1-d^2 x_1-d y_1 z_2 & 2 y_1-{x_2}^2 y_1-d x_2 z_3 \\
-x_1 x_2+d z_1 & x_1 y_2-d z_3 & -x_1 y_2+d z_3 & 0 & y_2 z_1-x_2 z_3
\\
 -x_2 (-2+{x_1}^2+{y_1}^2)
& 2 y_2-d y_1 z_1-d x_1 z_3 & -(-2+{x_1}^2+{y_1}^2) y_2 & -d (-2+{x_1}^2+{y_1}^2)
& -x_2 y_1 z_1+2 z_2-x_1 x_2 z_3 \\
 -d x_2 y_1-x_1
x_2 z_2+2 z_3 & -d x_1 x_2+2 z_1-d^2 z_1 & -d y_1 y_2+2 z_1-x_1 y_2 z_2
& 2 y_1-d^2 y_1-d x_1 z_2 & 2 x_1-x_1 {x_2}^2-d x_2 z_1 \\
 -x_2 y_1+d z_3 & y_1 y_2-d z_1 & -y_1 y_2+d z_1 & 0 & -x_2 z_1+y_2 z_3
\\
 -x_2 (-2+d^2+{z_2}^2) & 2 y_2-d^2 y_2-d x_2 z_2
& -y_2 (-2+d^2+{z_2}^2) & -d (-2+d^2+{z_2}^2) & -d x_2 y_2+2 z_2-{x_2}^2 z_2 \\
y_1 z_1-x_2 z_2 & -d x_2+x_1 z_1 & -y_2 z_2+y_1 z_3 & x_1 y_1-d
z_2 & -(x_2-z_1) (x_2+z_1)
\end{array}
\right)
\end{eqnarray*}
\end{tiny}
$$\det \bar{G_2} = \frac{\det G_2}{(1-d^2)^4{\det G_1}^2}$$

\subsection{}$\det G_2$
\begin{tiny}
\begin{eqnarray*}
\det G_2 &= &-d^2 (-x_1 x_2+x_2 y_1+x_1 y_2-y_1 y_2+d z_1-z_1 z_2-d z_3+z_2 z_3)^4(-x_1 x_2-x_2 y_1-x_1 y_2-y_1 y_2+d z_1+z_1 z_2+d z_3+z_2 z_3)^4\\
&&(-x_1 x_2 z_1-y_1 y_2 z_1+d {z_1}^2+x_2 y_1 z_3+x_1 y_2 z_3-d {z_3}^2)^2\\
&&(8 d^2-2 d^4-8 {x_1}^2+2 d^2 {x_1}^2-8 {x_2}^2+2 d^2 {x_2}^2+2 {x_1}^2 {x_2}^2+8 d x_1 y_1-2 d^3 x_1y_1- 2 d x_1 {x_2}^2 y_1-8 {y_1}^2+2 d^2 {y_1}^2+2 {x_2}^2 {y_1}^2+\\
&&8 d x_2 y_2-2 d^3 x_2 y_2-2d {x_1}^2 x_2y_2+2 d^2 x_1 x_2 y_1 y_2-2 d x_2 {y_1}^2 y_2-8 {y_2}^2+2 d^2 {y_2}^2+2{x_1}^2 {y_2}^2-2 d x_1 y_1 {y_2}^2+2 {y_1}^2 {y_2}^2+\\
&&2 d x_1 x_2 z_1-d^3 x_1 x_2 z_1-4x_2 y_1 z_1+2 d^2 x_2 y_1 z_1-4 x_1 y_2 z_1+2 d^2 x_1 y_2 z_1+2 d y_1y_2 z_1-d^3 y_1 y_2 z_1+8 {z_1}^2-6 d^2 {z_1}^2+d^4 {z_1}^2-\\
&&8 d^2 z_2+2 d^4 z_2+d x_1x_2 z_1 z_2-2 x_2 y_1 z_1 z_2-2 x_1 y_2 z_1 z_2+d y_1 y_2 z_1z_2+4 {z_1}^2 z_2-d^2 {z_1}^2 z_2+8 {z_2}^2-2 d^2 {z_2}^2-4 x_1 x_2 z_3+\\
&&2 d^2 x_1 x_2z_3+2 d x_2 y_1 z_3-d^3 x_2 y_1 z_3+2 d x_1 y_2 z_3-d^3 x_1 y_2 z_3-4y_1 y_2 z_3+2 d^2 y_1 y_2 z_3-2 x_1 x_2 z_2 z_3+d x_2 y_1 z_2 z_3+dx_1 y_2 z_2 z_3-\\
&&2 y_1 y_2 z_2 z_3+8 {z_3}^2-6 d^2 {z_3}^2+d^4 {z_3}^2+4 z_2 {z_3}^2-d^2z_2 {z_3}^2)\\
&&(8 d^2-2 d^4-8 {x_1}^2+2 d^2 {x_1}^2-8 {x_2}^2+2 d^2 {x_2}^2+2 {x_1}^2 {x_2}^2-8 d x_1y_1+2 d^3 x_1 y_1+2 d x_1 {x_2}^2 y_1-8 {y_1}^2+2 d^2 {y_1}^2+2 {x_2}^2 {y_1}^2-\\
&&8 d x_2 y_2+2d^3 x_2 y_2+2 d {x_1}^2 x_2 y_2+2 d^2 x_1 x_2 y_1 y_2+2 d x_2 {y_1}^2 y_2-8 {y_2}^2+2d^2 {y_2}^2+2 {x_1}^2 {y_2}^2+2 d x_1 y_1 {y_2}^2+2 {y_1}^2 {y_2}^2+\\
&&2 d x_1 x_2 z_1-d^3 x_1x_2 z_1+4 x_2 y_1 z_1-2 d^2 x_2 y_1 z_1+4 x_1 y_2 z_1-2 d^2 x_1 y_2
z_1+2 d y_1 y_2 z_1-d^3 y_1 y_2 z_1+8 {z_1}^2-6 d^2 {z_1}^2+d^4 {z_1}^2+\\
&&8 d^2 z_2-2 d^4z_2-d x_1 x_2 z_1 z_2-2 x_2 y_1 z_1 z_2-2 x_1 y_2 z_1 z_2-d y_1y_2 z_1 z_2-4 {z_1}^2 z_2+d^2 {z_1}^2 z_2+8 {z_2}^2-2 d^2 {z_2}^2+4 x_1 x_2 z_3-\\
&&2d^2 x_1 x_2 z_3+2 d x_2 y_1 z_3-d^3 x_2 y_1 z_3+2 d x_1 y_2 z_3-d^3 x_1y_2 z_3+4 y_1 y_2 z_3-2 d^2 y_1 y_2 z_3-2 x_1 x_2 z_2 z_3-d x_2 y_1z_2 z_3-d x_1 y_2 z_2 z_3-\\
&&2 y_1 y_2 z_2 z_3+8 {z_3}^2-6 d^2 {z_3}^2+d^4 {z_3}^2-4z_2 {z_3}^2+d^2 z_2 {z_3}^2)
\end{eqnarray*}
\end{tiny}

\subsection{}Terms of maximal degree in $\det G_3$
\begin{tiny}
\begin{eqnarray*}
h(\det G_3) &= &h(\det G_2)^6{\det G_1}^{-9}d^{30}w^3\bar{w}^3\\
&= &d^{66} (-x_1 x_2+x_2 y_1+x_1 y_2-y_1 y_2+d z_1-z_1 z_2-d z_3+z_2
z_3)^{15}\\
&&(-x_1 x_2-x_2 y_1-x_1 y_2-y_1 y_2+d z_1+z_1 z_2+d z_3+z_2
z_3)^{15}\\
&&(-x_1 x_2 z_1-y_1 y_2 z_1+d z_1^2+x_2 y_1 z_3+x_1 y_2
z_3-d {z_3}^2)^{12}\\
&&(2 x_1 x_2 y_1 y_2-d x_1 x_2 z_1-d y_1 y_2 z_1+d^2
{z_1}^2-d x_2 y_1 z_3-d x_1 y_2 z_3+d^2 {z_3}^2)^{12}\\
&&(x_1 x_2 y_1 y_2
z_1-d x_1 x_2 {z_1}^2-d y_1 y_2 {z_1}^2+d^2 {z_1}^3-x_1 x_2 y_1 y_2 z_3+d
x_2 y_1 {z_3}^2+d x_1 y_2 {z_3}^2-d^2 {z_3}^3)^3\\
&&(x_1 x_2 y_1 y_2 z_1-d
x_1 x_2 z_1^2-d y_1 y_2 z_1^2+d^2 {z_1}^3+x_1 x_2 y_1 y_2 z_3-d x_2 y_1
{z_3}^2-d x_1 y_2 {z_3}^2+d^2 {z_3}^3)^3
\end{eqnarray*}
\end{tiny}

\subsection{}$\det G_3$ with substitution $x_1=x_2=y_1=y_2=z_2=0$
\begin{tiny}
\begin{eqnarray*}
\det G_3|_{x_1=x_2=y_1=y_2=z_2=0} &= &(-2 + d)^{16}(-1 + d)^4d^{60}(1 + d)^4(2 + d)^{16}(-3 + d^2)^6({z_ 1} - {z_ 3})^{30} ({z_ 1} + {z_ 3})^{30}\\
&&({z_1}^2 - {z_ 1} {z_ 3} + {z_ 3}^2) ({z_ 1}^2 + {z_ 1} {z_ 3} + {z_3}^2)(-2 d^2 - 2 {z_ 1}^2 + d^2 {z_ 1}^2 - 2 {z_ 3}^2 + d^2 {z_ 3}^2)^{12}\\
&& (-3 d^2 - {z_ 1}^2 + d^2 {z_ 1}^2 + {z_ 1} {z_ 3} - d^2 {z_ 1} {z_ 3} - {z_ 3}^2 + d^2 {z_ 3}^2)^2(-3 d^2 - {z_ 1}^2 + d^2 {z_ 1}^2 - {z_ 1} {z_ 3} + d^2 {z_ 1} {z_ 3} - {z_ 3}^2 + d^2 {z_ 3}^2)^2
\end{eqnarray*}
\end{tiny}

\subsection{}$\det G_1^{F_{0,3}}$ with substitution $x_{a_1} = x_{a_1,a_2} = 0$ for all variables of the form $x_{a_1}$ and $x_{a_1,a_2}$
\begin{tiny}
\begin{eqnarray*}
\det G_1^{F_{0,3}}|_{x_{a_1} = x_{a_1,a_2} = 0} &= &-(d-x_{1,2,3})(d+x_{1,2,3})\\
&&(x_{1,2,-2}x_{1,-1,3}x_{1,-1,-2}+x_{1,3,-3}x_{1,-1,2}x_{1,-1,-3}-x_{1,2,-3}x_{1,-1,3}x_{1,-1,-3}-\\
&&x_{1,-1,-2}x_{1,-1,-2}x_{1,-2,3}-x_{1,2,-2}x_{1,3,-3}x_{1,-2,-3}+x_{1,2,-3}x_{1,-2,3}x_{1,-2,-3})^2
\end{eqnarray*}
\end{tiny}
\end{landscape}

\vspace{5ex}
\begin{minipage}[b]{0.45\linewidth}

Department of Mathematics\newline\indent
The George Washington University\newline\indent
Washington, DC 20052, USA\newline\indent
\texttt{przytyck@gwu.edu}

\end{minipage}
\hspace{3ex}
\begin{minipage}[b]{0.45\linewidth}

Department of Mathematics\newline\indent
Harvard University\newline\indent
Cambridge, MA 02138, USA\newline\indent
\texttt{xzhu@fas.harvard.edu}

\end{minipage}


\begin{thebibliography}{1}

\bibitem{Che}
Q.~Chen and J.~H.~Przytycki,
The Gram determinant of the type B Temperley-Lieb algebra,\  
e-print, \url{http://arxiv.org/abs/0802.1083v2}.

\bibitem{De}
E.~Deutsch and B.~E.~Sagan, Congruences for Catalan and Motzkin numbers 
and related sequences,  {\it J. Number Theory}, 117, 2006, 191--215.

\bibitem{DiF}
P.~Di Francesco, Meander determinants, {\it Comm. Math. Phys.}, 191,
1998, 543-583.

\bibitem{Lic-1}
W.~B.~R.~Lickorish, Invariants for 3-manifolds from
the combinatorics of the Jones polynomial, {\it Pacific Journ. Math.},
149(2), 337--347, 1991.

\bibitem{Lic-2}
W.~B.~R.~Lickorish.
\newblock {\em An introduction to knot theory}, volume 175 of {\em Graduate
  Texts in Mathematics}.
\newblock Springer-Verlag, New York, 1997.

\bibitem{MS}
P.~P.~Martin, H.~Saleur,
On an Algebraic Approach to Higher Dimensional Statistical Mechanics,
{\it Commun. Math. Phys.} 158, 1993, 155-190.

\bibitem{Sch} F.~Schmidt, Problems related to type-A and type-B
matrices of chromatic joins, {\it Advances in Applied Mathematics},
32:(380--390), 2004.

\bibitem{Sim} R.~Simion,
Noncrossing partitions, {\it Discrete Math.}, 217, 2000, 367-409.

\bibitem{Sha} L.~Shapiro, Personal communications (email 2 Sep 2008), 
and talk at GWU combinatorics seminar, 25 Sep 2008.

\bibitem{Sta} R.~P.~Stanley.
\newblock {\em Enumerative combinatorics, vol. 2}.
\newblock Cambridge University Press, New York, 1999.

\bibitem{Wes} B.~W.~Westbury,
\newblock The representation theory of the {T}emperley-{L}ieb algebras.
\newblock {\em Math. Z.}, 219(4):539--565, 1995.

\end{thebibliography}
\end{document}